\documentclass[a4paper,10pt]{amsart}
\usepackage[T1]{fontenc}
\usepackage[utf8]{inputenc}
\newtheorem{thm}{Theorem}
\newtheorem{prop}{Proposition}
\newtheorem{lemma}{Lemma}
\newtheorem{cor}{Corollary}
\title{Some inequalities for central moments of matrices}
\author{Zoltán Léka}
\address{Budapest Business School \\ 1149 Budapest \\ Buzog\'any u. 11--13.}
\email{leka.zoltan@renyi.mta.hu}

\thanks{This study was partially supported by Hungarian NSRF (OTKA) grant
no. K104206 }

\subjclass[2010]{Primary 47A63, 47C15, 46L53,  ; Secondary 15A60, 60B99}
\keywords{variance, standard deviation, central moments, conditional expectation, density, extreme point, spread}

\date{}

\begin{document}

\begin{abstract}
   In this paper we shall study noncommutative central moment inequalities with a main focus on whether the commutative 
   bounds are tight in the noncommutative case, or not. We prove that the answer is affirmative for the fourth central moment
   and several particular results are given in the general case. As an application, we shall present some lower estimates of the spread of Hermitian and
   normal matrices as well.
\end{abstract}


\maketitle

\section{Introduction}
    Let $X$ be a random variable on a probability space $(\Omega, P)$. Then its $p^{\rm th}$
    (fractional) central moments are defined by the formula
     $$ \mu_p(X)= \int_\Omega \left|X - \int_\Omega X \: dP \right|^p \: dP.$$
    The most studied noncommutative analogue of these quantities is the noncommutative variance or quantum variance. Let $M_n(\mathbb{C})$ be
    the algebra of $n \times n$ complex matrices. Whenever
    $\Phi \colon M_n(\mathbb{C}) \rightarrow  M_m(\mathbb{C})$ is a positive  unital linear map,
    the variance of a matrix $A$ can be defined as $\Phi(A^*A) - \Phi(A)^*\Phi(A).$ 
    For several interesting properties of this variance, we refer the reader to Bhatia's book \cite{Bh}. 
    For instance, special choices of $\Phi$ and applications of variance estimates provided simple new proofs of spread estimates of normal
    and Hermitian matrices as well, see \cite{BSh} and \cite{BSh2}. On the other hand, the first sharp estimate of the noncommutative variance appeared
    in K. Audenaert's paper \cite{A1} in connection with the Böttcher--Wenzel commutator estimate. For several different proofs of his result, we refer to
    \cite{BG}, \cite{BSh2} and \cite{R2}. Recently, extremal properties of the quantum variance were studied in \cite{PT}.
    
    It is simple to see that if $\omega$ is a state (i.e. positive linear functional of norm $1$) of the algebra 
    $M_n(\mathbb{C}),$ then one has the upper bound
     $$ \omega(|A - \omega(A)|^2) = \omega(|A|^2) - |\omega(A)|^2 \leq \|A\|^2$$ (see \cite[Theorem 3.1]{BSh} for positive linear maps).
    A careful look of the previous inequality says that the noncommutative variance cannot be larger than the ordinary variance 
    of random variables.
    In fact, if $X$ is a Bernoulli variable, that is, $P(X = 0) = p$ and $P(X = 1) = 1-p$  $(0\leq p \leq 1),$ then 
                $ \mu_2(X) \leq 1/4$ holds. Furthermore, for any (complex-valued) random variable $Z \colon \Omega \rightarrow \mathbb{C}$ the inequality 
    \begin{eqnarray*}
       \sqrt{\mu_2(Z)} &\leq& 2 \max \{ \sqrt{ \mu_2(X)} \colon X \mbox{ Bernoulli random variable } \}  \|Z\|_\infty \\ 
                       &=&  \|Z\|_\infty 
    \end{eqnarray*}
    readily follows, see \cite[Theorem 7]{A1} in the discrete case and \cite[Theorem 2]{L2} in the general case, for instance. 
     Furthermore, one can have the following upper bound for the  $n^{\rm th}$ ($n \in \mathbb{N}$) central moment 
    of a normal element $A$ in matrix and $C^*$-algebras 
     $$ \sqrt[n]{\omega(|A-\omega(A)|^n)} 
     \leq 2 \max \{ \sqrt[n]\mu_n(X) \colon X \mbox{ Bernoulli random variable } \} \min_{\lambda \in \mathbb{C}} \|A-\lambda\|,$$
    see \cite[Theorem 2]{L2}. 
   
    Our main motivation is to provide sharp upper bounds on the non-commutative central moments of arbitrary matrices and to decide
    whether the noncommutative dispersion can be larger than that of the commutative one. 
    Now we are able to tackle the problem only for $ 1 \leq p \leq 2,$ $p=4$ 
    and for the tracial state if $1 \leq p < \infty$. For some recent results on the complementary
    upper bounds on fourth central moments of matrices,
    the reader might see the paper \cite{K}.  
    We note that a very similar phenomena was observed by K. Audenaert \cite{A2}. He proved that the asymmetry
    of the quantum relative entropy essentially cannot be larger 
    than that of two Bernoulli distributions.    
    
    The paper is organized as follows. In the next section we prove an inequality for the fourth moments of  partial isometries and positive
    linear maps given by unit vectors. To set free these assumptions, we shall apply a dilation method and
    a rank-estimate on the extreme points of convex sets of density matrices. In the last section of the paper, we shall produce some general results on $p^{\rm th}$ moments of matrices, determined by the tracial state. 
    As an application, we shall present several lower bounds for the spread of normal and Hermitian matrices.
 
 \section{General moment inequalities} 
 
\subsection{A moment estimate of partial isometries} 
 Let $X$ be a Bernoulli 
 random variable with parameter $p.$
 Then one has clearly the inequality $\mu_4(X)  \leq {1 \over 12},$ while $\mu_4(Z)  \leq {4 \over 3} \min_{\lambda \in \mathbb{C}}\max_i |z_i - \lambda|^4$ comes true for any finite--valued 
 random variable $Z$. From a geometric point of view, the  quantity $\min_{\lambda \in \mathbb{C}}\max_i |z_i - \lambda|$ is the radius of the smallest enclosing circle of the values of $Z.$  
  For several inequalities in connection with it and vector norms, the reader might see \cite[Section 4]{A1}.
  
  Our first result 
 gives the corresponding noncommutative moment estimate for partial isometries. 
 Recall that an $n \times n$ matrix $V$ is partial isometry if $V$ is an isometry on the orthogonal complement of its kernel.
 A useful characterization says that $V$ is a partial isometry if and only if $V^*V$ is an orthogonal projection (to the subspace $(\mbox{ker } V)^\perp$), or, which is the 
 same, $V = VV^*V$ (see \cite{H}, \cite[page 95]{Pe}). Hence $V^*$ is a partial isometry and $VV^*$ is an orthogonal projection as well.  
 
 We start with a technical lemma. 
 
 \begin{lemma}
    Let  $x_1,x_2,x_3,x_4 \in \mathbb{R}.$  Then 
      $$ \max_x \: ( 2x_1^2x_3-2x_1x_4 +x_2^2 )= 1$$
      subject to the constraints
           $0 \leq x_3 \leq x_2 \leq 1$ and 
           $x_1 \leq x_4 +  \sqrt{(1-x_2^2)(x_2^2 - x_3^2)}.$
 \end{lemma}

 \begin{proof} With a change of variables $y_2 = \sqrt{1-x_2^2,}$ $y_3 = \sqrt{x_2^2-x_3^2}$ and $y_1 = x_1,y_4 = x_4$, we have
      $ 2x_1^2x_3-2x_1x_4 +x_2^2 - 1 = 2y_1^2\sqrt{1-y_2^2-y_3^2}- 2y_1y_4-y_2^2.$ Notice that the last function is convex in $y_1$, hence it attains its maximum
      when $y_1$ is the largest, i.e. $y_1 = y_4 + y_2y_3.$ Therefore, it is enough to prove the general statement 
      \begin{eqnarray*}
      \max \: G(y_1,y_2,y_3,y_4) &=& \max \: 2y_1^2\sqrt{1-y_2^2-y_3^2}- 2y_1y_4-y_2^2 = 0 \\ && \mbox{s.t. } y_1 = y_4 + y_2y_3.
      \end{eqnarray*}
      
     First, we compute the extrema of $G$ when $(y_2,y_3,y_4)$ is in the open cylinder $\mathbb{D} \times \mathbb{R},$ where $\mathbb{D}$ denotes
     the open unit disk of the plane. To do this, let us find the  constrained critical points of the Lagrangian
     $\mathcal{L}(y, \lambda) = G(y) - \lambda c(y),$ where the constraint function is $c(y) = y_1 - y_4 - y_2y_3.$ 
     A little calculation gives for the gradient equation $\nabla \mathcal{L}(y,\lambda) = 0$ that
     \begin{eqnarray*}
      \begin{split}
       -\lambda - 2y_4 + 4y_1\sqrt{1-y_2^2 - y_3^2} &= 0 \\
        \lambda y_3 - {2y_1^2y_2 \over \sqrt{1-y_2^2-y_3^2}} - 2y_2 &= 0 \\
       \lambda y_2 - {2y_1^2y_3 \over \sqrt{1-y_2^2-y_3^2}} &= 0\\        
         \lambda-2y_1 &= 0 \\ 
        -y_1+y_2y_3+ y_4 &= 0. 
      \end{split}
     \end{eqnarray*}
     
     To solve this system, note that if $y_1 = \lambda /2 = 0$ then $y_2 = y_4 = 0$ and $-1 < y_3 < 1.$
     From $\lambda = y_4 = 0,$ it is simple to check that $G \leq 0.$ Indeed, $y_1 = y_2y_3$ and
     $$ 2y_2^2y_3^2 \sqrt{1-y_2^2-y_3^2} -  y_2^2 \leq 0, $$
     because of the inequalities
     \begin{eqnarray*}
        \begin{split}
            2y_3^2 \sqrt{1-y_2^2-y_3^2} \leq 2y_3^2 \sqrt{1-y_3^2} \leq  2 \sqrt{y_3^2(1-y_3^2)} \leq 1.   \\   
        \end{split}
      \end{eqnarray*}
    
    On the other hand, if $\lambda \neq 0,$ from the third and fourth equation $$  2y_2  \sqrt{1-y_2^2-y_3^2} =  \lambda y_3.$$
     Substitute this to the second one and we obtain that
      $$  y_2  (1-y_2^2-y_3^2) = y_1^2y_2  + y_2\sqrt{1-y_2^2-y_3^2}.  $$
     Since $ 1-y_2^2-y_3^2 < y_1^2  + \sqrt{1-y_2^2-y_3^2},$ if $(y_2,y_3) \in \mathbb{D}$ and $y_1 \neq 0,$ it follows that $y_2 = 0.$ Clearly, 
     $y_2 = y_3 = 0$ and $y_1 = y_4 = \lambda /2$ hold. The corresponding Hessian of $\mathcal{L}$ at a stationary point $(y_*, \lambda_*) = (t,0,0,t,2t)$
     is  
     $$ \nabla_{yy} \mathcal{L}(y_*,\lambda_*) = \begin{bmatrix} 4 & 0 & 0 & -2 \\ 0 &-2(t^2+1)  & 2t & 0 \\ 0& 2t & -2t^2 &0 \\ -2&0&0&0   \end{bmatrix}. $$
     Now let us consider an $y_{4*}$-sections of the cylinder  $\mathbb{D} \times \mathbb{R};$ that is, add the constraint $y_4 = t \: (\neq 0)$ to the optimization problem.
     Then $c_2(y) = y_4 - t$ and $\nabla c_2(y_*) = [0,0,0,1]^*.$  
      Note that $\nabla c(y_*) = [1,0,0,-1]^*,$ hence the tangent plane of the constraints at $y_*$ is $$ \mathfrak{T}(\lambda_*) := \{ w \colon w^* \nabla c(y_*) = 0 \mbox{ and } w^* \nabla c_2(y_*) = 0  \} 
     = \{ [0 , w_2 , w_3 , 0]^* \colon w_i \in \mathbb{R} \}.$$
     Furthermore, we obtain  
       $$ w^*\nabla_{yy} \mathcal{L}(y_*,\lambda_*)w = -2t^2w_2^2 - 2(w_2 - tw_3)^2 < 0, \qquad   0 \neq w \in \mathfrak{T}(\lambda_*).$$
   Thus $y_*$ is a strict local maximum of $G$ subject to $c$ and $c_2$, see \cite[Theorem 12.6]{Nu},
     and $G(y_*) = 0.$
      
      For $y_4 \neq 0,$ all $y_4$-sections of the cylinder  $\mathbb{D} \times \mathbb{R}$ contain exactly one local maximum point, hence $0$ is the global
      maximum of $G$ on its domain (s.t. the constraint).
    
  %
  %
  \end{proof} 
 
\begin{thm} Let $V$ be a partial isometry in $M_n(\mathbb{C}).$ Let $Q \in M_n(\mathbb{C})$ be a rank-one orthogonal 
projection. Then
  $$ {\rm Tr }  \: [Q|V- {\rm Tr } \: [QV]|^4] \leq  {4 \over 3}.$$
  Moreover, if the equality holds then $|V|Q = Q|V|.$ 
 \end{thm} 
 
 \begin{proof} 
   Without loss of generality, we can assume that $\alpha := \mbox{Tr } QV =  \mbox{Tr } QV^* \geq 0.$
    Let $V = UP$ be a polar decomposition of $V,$ where $P = V^*V$ is an orthogonal projection and $U$ is unitary. 
    Choose a unit vector $x$ such that $Q = x^*x.$ First, one has that
   \begin{eqnarray*}
     \begin{split}
      \mbox{Tr } [Q|V- \alpha I|^4] &= \mbox{Tr } [Q( V^*V - \alpha V^*V^2 - \alpha V^*VV^* + \alpha^2 V^*V \\
                    & - \alpha VV^*V + \alpha^2 V^2 + \alpha^2 VV^* - \alpha^3V \\
                    & - \alpha {V^*}^2V + \alpha^2 V^*V + \alpha^2 {V^*}^2 - \alpha^3 V^* \\
                    &   + \alpha^2 V^*V - \alpha^3 V - \alpha^3 V^* + \alpha^4 I)] \\ 
                    & \hspace{-2.6cm} \mbox{and applying the identities } V^*VV^* = V^* \mbox{ and  } VV^*V = V, \\ 
                   &=  \|Px\|^2 - 2 \alpha  \mbox{Re} \langle V^*V^2x,x \rangle +  \alpha^2(3\|V^*Vx\|^2+\|VV^*x\|^2-2) \\ 
                                 &\quad + 2\alpha^2\mbox{Re}\langle V^2x,x \rangle-  3 \alpha^4 \\  
                & \hspace{-2.6cm} \mbox{and since } \|V^*Vx\|^2+\|VV^*x\|^2 \leq 2, \\                   
                   &\leq  \|Px\|^2 - 2 \alpha \mbox{Re} \langle Vx,Px \rangle + 2\alpha^2\|Px\|^2  + 2\alpha^2|\mbox{Re}\langle V^2x,x \rangle|-  3 \alpha^4.          
     \end{split}
   \end{eqnarray*}
    
    Next, from the  Cauchy--Schwarz inequality $$ |\mbox{Re} \langle Vx,V^*x \rangle| = |\mbox{Re } \langle PUPx,PU^*x \rangle| 
     \leq \|PUPx\|\|PU^*x\| \leq \|PUPx\|.$$
       Therefore, we obtain    
      \begin{eqnarray*}
     \begin{split}
       \mbox{Tr } [Q|V- \alpha I|^4]  &\leq  \|Px\|^2  + 2 \alpha^2\|Px\|^2-  3\alpha^4 - 2\alpha {\mbox{Re} \: \langle UPx,  Px \rangle} + 2\alpha^2 \|PUPx\| \\
                                   &\leq  \max_{0 \leq \alpha \leq 1} 1+ 2 \alpha^2-  3\alpha^4  \\
                                   &\quad + \max ( \|Px\|^2 -1 - 2\alpha {\mbox{Re} \: \langle UPx,  Px \rangle} + 2\alpha^2 \|PUPx\|).
     \end{split}
   \end{eqnarray*}
       It is simple to check that $$ \max_{0 \leq \alpha \leq 1} 1+ 2 \alpha^2-  3\alpha^4 = {4 \over 3}.$$ 
       For the remaining part of the previous inequality, from the  Cauchy--Schwarz inequality we have the following constraint for $0 \leq \alpha$ 
         \begin{eqnarray*}
     \begin{split}
               \alpha        &=  \mbox{Re } \langle  UPx,x \rangle  \\ 
                             &=  \mbox{Re } \langle  PUPx,Px \rangle + \mbox{Re } \langle  P^\perp UPx,P^\perp x \rangle \\
                             &\leq \mbox{Re } \langle UPx, Px \rangle + \|P^\perp UPx\|\|P^\perp x\| \\
                             &= \mbox{Re } \langle UPx, Px \rangle + ( \|UPx\|^2-\|PUPx\|^2)^{1/2}(1-\|Px\|^2)^{1/2} \\                       
                             &= \mbox{Re } \langle UPx, Px \rangle + ( \|Px\|^2-\|PUPx\|^2)^{1/2}(1-\|Px\|^2)^{1/2}.                           
     \end{split}                        
   \end{eqnarray*}      
     Hence we can apply Lemma 1 to obtain that  
        $$  2\alpha^2 \|PUPx\| - 2\alpha {\mbox{Re} \: \langle UPx,Px \rangle} + \|Px\|^2  \leq 1.$$
      Thus the inequality
      $$\mbox{Tr } [Q|V- \alpha I|^4] \leq {4 \over 3} $$ follows.
      
      Note that when the equality occurs $\|Px\|^2 = 1$ must hold. This means that
      $Px = x,$ which is the same as $x \in \mbox{ran } P,$ or, $QV^*V = VV^*Q.$ 
   \end{proof}

  \medskip
  
  \noindent {\bf Example 1.} We give a matrix example, which is not normal, to see that the previous inequality is tight. Set the partial isometry 
   $$ V = {1 \over \sqrt{3} }\begin{bmatrix} 1 & 1 & 1 & 0 \cr 1 & -1/2 & -1/2 & 1 \cr 1 & -1/2 & -1/2 &  \hspace{-.2cm} -1  \cr 0 & 0 & 0 &1    \end{bmatrix}$$
   and define
   $$ Q = \begin{bmatrix} 1 & 0 & 0 & 0 \cr 0 & 0 & 0 & 0  \cr 0 & 0 & 0 & 0 \cr 0 & 0 & 0 & 0  \end{bmatrix}. $$
   In fact, $V$ is a partial isometry because $V^*V$ is an orthogonal projection. Moreover, the spectrum of $V$ is $\sigma(V) = \{1,1/\sqrt{3},0,-1 \}$, hence $\min_{\lambda \in \mathbb{C}} \|V-\lambda I_n \| = 1.$  
   Furthermore, it is simple to see that  
    $$  {\rm Tr } \: [Q|V- {\rm Tr } \: [QV]|^4] =  {4 \over 3}.  $$
 
 \subsection{Convex sets of density matrices.} Let $A_1, A_2, \hdots, A_k \in M_n(\mathbb{C})$ be Hermitian matrices and let $\alpha_1, \alpha_2, 
 \hdots, \alpha_k$ be real numbers. Let us consider the  convex, compact set 
 $$ \mathcal{D}(A_1, A_2, \hdots, A_k) := \left\{ X \geq 0 \colon \mbox{ Tr } X = 1 \mbox{ and  Tr } [XA_i] = \alpha_i, \: i = 1, 2, \hdots, k \right\}. $$
  Note that $\mathcal{D}(A_1, A_2, \hdots, A_k) = \mathcal{D}(A_1-\alpha_1 I, A_2-\alpha_2 I, \hdots, A_k-\alpha_k I). $ 
 The geometry of $\mathcal{D}(A_1, A_2, \hdots, A_k)$ is strongly related to that of the elliptope; 
 i.e. the set of real $n \times n$ symmetric positive semidefinite matrices with an all-one diagonal (briefly, correlation matrices) \cite{CT}, \cite[Chapter 31.5]{DL}. 
 Additionally, we used the set $\mathcal{D}(A_1, A_2, \hdots, A_k)$ to provide a 
 description of the extreme non-commutative covariance matrices associated to Hermitian tuples (see \cite{L}).
  
  We recall that whenever $D$ is an extreme point of  $\mathcal{D}(A_1, A_2, \hdots, A_k),$ one has the rank estimate \cite[Corollary 1]{L}
 \begin{equation}
   \mbox{rank } D \leq \sqrt{k+1}.  
 \end{equation}
We remark that the proof of the previous inequality is closely related to a method invented by C.K. Li and B.S. Tam \cite{CT} in order to describe 
the extreme correlation matrices. 

 Turning back to moment inequalities, Audenaert's theorem \cite{A1} on the (quantum) standard deviation states that for any $A \in M_n(\mathbb{C})$ there exists a rank-one orthogonal projection $P$ such that
 $$ \max_{D \geq 0, \\ {\rm Tr} D = 1} \: ({\rm Tr } \: [D|A- {\rm Tr } \: [DA]|^2])^{1/2} =  ({\rm Tr } \: [P|A- {\rm Tr } \: [PA]|]^2)^{1/2} =  \min_{\lambda \in \mathbb{C}} \|A - \lambda I\| .$$
  This result was proved directly in \cite{BS} and in \cite{BG} for $C^*$-algebras by means of a characterization of the Birkhoff--James orthogonality in matrix and operator algebras, respectively. 
 
 Throughout the paper we say that an $n\times n$ matrix $D$ is a
 density if $D$ is positive semidefinite and Tr $D = 1.$
 Exploiting the aforementioned rank estimate, now we can prove the following   

\begin{thm}
  Let $D \in M_n(\mathbb{C})$ be a density. For any $1 \leq p < \infty$ and $A \in M_n(\mathbb{C}),$ 
  there exists a rank-one orthogonal projection $P \in M_n(\mathbb{C})$ such that 
  $$ {\rm Tr } \: [D|A- {\rm Tr } \: [DA]|^p] =  {\rm Tr } \: [P|A- {\rm Tr } \: [PA]|^p]. $$
\end{thm}

\begin{proof}
    Without loss of generality we can assume that $\mbox{Tr }[DA] = \alpha$ is real, 
    hence $\displaystyle \mbox{Tr }  D{A + A^* \over 2} = \alpha$ holds as well.   Let us introduce the convex set
  \begin{eqnarray*}
   \begin{split}
    \mathcal{D}(|A - \alpha I_n|^p,{A+A^* \over 2}) := \Bigl\{ X \mbox{ density } \colon \mbox{Tr } [X|A - \alpha I_n|^p] 
                                                          &=  \mbox{Tr }[D|A - \alpha I_n|^p] \\ & \mbox{ and  Tr } X{A + A^* \over 2} = \alpha \Bigr\}.
   \end{split}                                                       
  \end{eqnarray*}
   Obviously, $\mathcal{D}(|A - \alpha I_n|^p,{A+A^* \over 2})$ is non-empty.
   Relying upon the inequality $(1),$  the rank of extreme points of $\mathcal{D}$ is at most $\sqrt{3}.$ 
   Since any rank-$1$ density is an orthogonal projection, the proof is complete. 
\end{proof}

 Now we can prove the main theorem of the section.

\subsection{A 4-order moment estimate}
  
 \begin{thm} Let $A \in M_n(\mathbb{C})$ and let $D \in M_n(\mathbb{C})$ be a density matrix. Then
  $$ {\rm Tr } \: [D|A- {\rm Tr } \: [DA]|^4] \leq  {4 \over 3} \min_{\lambda \in \mathbb{C}} \|A-\lambda I_n \|^4.$$
 \end{thm} 
  \begin{proof} 
   Without loss of generality, we can assume that $\|A\| = 1$ holds. Now form the partial isometry (see \cite{H})
    $$ V = \begin{bmatrix} A & (I - AA^*)^{1/2} \cr  0 & 0 \end{bmatrix} \quad \in M_{2n}(\mathbb{C}) .$$
   Let
    $$ \tilde{D}:= \begin{bmatrix} D & 0 \cr  0 & 0 \end{bmatrix},$$
    which is a density matrix, of course. Then a straightforward calculation gives that
     $$ |V- {\rm Tr } \: [\tilde{D}V]|^4 =  \begin{bmatrix} |A-{\rm Tr } \: [DA]|^4 + X & \quad \qquad * \quad \qquad \cr * & * \end{bmatrix}, $$
   where $X = (A-{\rm Tr } \: [DA])^*(I - AA^*)(A-{\rm Tr } \: [DA]) \geq  0,$ hence
      $$ |A- {\rm Tr } \: [DA]|^4 \leq \begin{bmatrix} I & 0 \cr  0 & 0 \end{bmatrix} |V- {\rm Tr } \: [\tilde{D}V]|^4 \begin{bmatrix} I & 0 \cr  0 & 0 \end{bmatrix}. $$
  Therefore the inequality 
    $$ {\rm Tr } \: [D|A- {\rm Tr } \: [DA] I_n|^4] \leq {\rm Tr } \: [\tilde{D}|V- {\rm Tr } \: [\tilde{D}V]I_{2n}|^4]$$
    follows. 
   Additionally, relying on Theorem 2, one can assume that $\tilde{D}= xx^*$ is a rank-one orthogonal projection with some unit vector  $x.$
   Then Theorem 1 immediately gives that
       \begin{eqnarray*}
     \begin{split}
        \mbox{Tr } [D|A- \mbox{Tr} [DA]|^4] &\leq  \mbox{Tr }[ \tilde{D}|V- \mbox{Tr} [\tilde{D}V]|^4] \\
                                            &\leq {4 \over 3} \|V\|^4 \\
                                            &= {4 \over 3} \|A\|^4. 
     \end{split}
     \end{eqnarray*} 
   Changing $A$ to $A - \lambda I,$ we get the proof of the statement.  
   \end{proof}

 Surprisingly, the next example shows that if $A$ is not normal the previous upper bound is not necessarily sharp. Please, compare it with Example 1 and Audenaert's theorem
 on the noncommutative variance.

\noindent {\bf Example 2.} Let $A$ denote the Jordan block $\begin{bmatrix}
                                                           1 & 1 \cr 0 & 1 
                                                          \end{bmatrix}
.$ We  calculate the value of 	
     $$ \mu_4(A) := \max  \: \{ {\rm Tr } \: [D|A- {\rm Tr } \: [DA]|^4]  \colon 0 \preceq D \in M_2(\mathbb{C}) \mbox{ and Tr } D=1 \}. $$
Indeed, from Theorem 2 we can find a projection $P = zz^*,$ $z^* = \begin{bmatrix} z_1 & z_2 
                                                                 \end{bmatrix} \in \mathbb{C}^2$ and $|z_1|^2 + |z_2|^2 =1 ,$ such that
      $$ \mu_4(A) = {\rm Tr } \: [P|A- {\rm Tr } \: [PA]|^4] . $$
Then a little computation gives that
 \begin{eqnarray*}
   \begin{split}
    {\rm Tr } \: [P|A- {\rm Tr } \: [PA]|^4]  =& |z_1|^6|z_2|^4 + |z_1|^4|z_2|^2 +|z_1|^2|z_2|^4 \\ 
                                          & - 2|z_1|^2|z_2|^2(2|z_1|^2|z_2|^2 + 1) + |z_2|^2(1+|z_1|^2|z_2|^2)^2 \\
                                          =& 4 |z_2|^6-3 |z_2|^8 =: p(|z_2|).  
   \end{split}
 \end{eqnarray*}
Note that $\max_{|z_2| \leq 1} p(|z_2|) = p(1) = 1.$ Furthermore, it is simple to check that  $\min_{\lambda \in \mathbb{C}} \|A-\lambda \| = \|A-I_2\| = 1.$ 
Hence we get that
   $$   \mu_4(A) = \min_{\lambda \in \mathbb{C}} \|A-\lambda\|^4. $$
However,  from  \cite[Theorem 4]{L2}  we have $$ \mu_4(A) = {4 \over {3}} \min_{\lambda \in \mathbb{C}} \|A-\lambda\|^4 $$
for any normal $A.$

\medskip

Here we make a direct application of Theorem 3 to obtain a lower bound for the spread of normal and Hermitian matrices.
If $A$ is an $n \times n$ matrix and $\lambda_i(A)$ $(1 \leq i \leq n)$
denote its eigenvalues then the spread of $A$ is
 $$ \mbox{spd}(A) = \max_{i, j} |\lambda_i(A) - \lambda_j(A)|. $$ 
 Spread estimates were initiated by L. Mirsky in his seminal papers \cite{M1} and \cite{M2}. 
 After that several author provided upper and lower bounds for it, see
 \cite{BH}, \cite{BSh2}, \cite{JKW} and \cite{MK}, for instance, and the references therein.
 
For a normal $A$ the spectral theorem gives that $\min_{\lambda \in \mathbb{C}} \|A-\lambda I \|=r_A,$ 
where $r_A$ denotes the radius of the smallest
disk that contains the eigenvalues of $A.$ 
Jung's theorem on the plane says that if $\mathcal{F}$ is a finite set of points of diameter $d$ then $\mathcal{F}$
must be contained in a closed disk of radius $d/\sqrt{3},$  see \cite[Chapter 16]{RT}. Hence, for any normal $A,$ 
 \begin{eqnarray}
   \min_{\lambda \in \mathbb{C}} \|A-\lambda I \|  \leq {1 \over \sqrt{3}} {\rm spd}(A)  
 \end{eqnarray}
(see \cite[p. 1567--1568]{BSh}).

\begin{cor}
  Let $A \in M_n(\mathbb{C})$ be a normal matrix. Then
  $$ {\rm Tr } \:[ D|A- {\rm Tr} \: [DA] |^4] \leq {4 \over 27} {\rm spd }(A)^4. $$
  Moreover, if $A$ is Hermitian then
   $$ {\rm Tr } \: [D|A- {\rm Tr} \: [DA]|^4] \leq { {\rm spd }(A)^4 \over 12}.$$
 \end{cor}
 
 For a different proof of the last statement, the reader might see \cite[p. 169 Remark]{L2}, \cite[Theorem 3]{K} and \cite[Theorem 2]{L2}
 for the normal case in Theorem 3.
 
 We recall that the quantity $\Delta(A) \equiv \min_{\lambda \in \mathbb{C}} \|A-\lambda I \|$  appeared in Stampfli's well-known result \cite{St} for the derivation norm
 $$2\Delta(A) = \max_{\|X\|=1} \|AX-XA\|,$$ 
 while the diameter of the unitary orbit of $A$ is given by the formula $$2\Delta(A) = \max \{\|A - UAU^*\| \colon U \mbox{ is unitary} \},$$
 see \cite{BS}. 
 Hence any lower estimate of $\Delta(A)$ in terms of central moments might have its own interest. In the case of the noncommutative variance,  this
 method was first exploited by R. Bhatia and R. Sharma in a series of papers \cite{BSh}, \cite{BSh2}. 
 Choosing different density matrices in the variance inequality, they got several interesting inequalities for $\Delta(A)$ and the spread of $A$, as well. 
 Briefly, the idea of non-commutative variance estimates turned out be fruitful and led to simple new proofs of known spread estimates, including Mirsky's and Barnes--Hoffman's
 lower bounds (see \cite{BSh} for details). 

\subsection{Remark}
 
  Let $\omega$ be a positive linear functional of $M_n(\mathbb{C}).$ 
  Then the map $A \mapsto \omega(|A|^p)^{1/p}, p \neq 2,$ is not a  norm on $M_n(\mathbb{C}),$ because the triangle inequality fails, in general.
  However, the monotonicity statement 
  $$ \omega(|A|^p)^{1/p} \leq  \omega(|A|^{p'})^{1/{p'}}$$
  clearly holds for all $1 \leq p \leq p' < \infty.$  In fact, $\omega(A) = \mbox{Tr } DA$ with some $D \geq 0$ and Tr $D=1$. Furthermore, 
  we can assume that $|A| = (A^*A)^{1/2}$ is diagonal, hence the discrete Hölder-inequality gives that
   $$ \left(\sum_{i=1}^n d_{ii} a_{ii}^p \right)^{1/p}  \leq  \left(\sum_{i=1}^n d_{ii} a_{ii}^{p'} \right)^{1/{p'}},$$
  which is exactly what we need. Moreover, $$\omega(|A- \omega(A)|^2)^{1/2} = (\omega(|A|^2) - |\omega(A)|^2)^{1/2} \leq \omega(|A|^2)^{1/2} \leq \|A\|.$$
  Therefore, we get for any $1 \leq p \leq 2$ and $A \in M_n(\mathbb{C})$ that 
   $$ \omega(|A- \omega(A)|^p)^{1/p} \leq \min_{\lambda \in \mathbb{C}} \|A-\lambda\|.$$ 
   
   Note that a simple calculus gives that  
    $$2 \max \{ (\mathbb{E}|X - \mathbb{E}(X)|^p)^{1/p}  \colon X \mbox{ Bernoulli random variable }\} = 1 $$ if $1 \leq p \leq 2,$ hence the commutative bound turns out to be
    a tight bound in the non-commutative case as well. Actually, set $A = \begin{bmatrix}
                                                           1 & 1 \cr 0 & 1 
                                                          \end{bmatrix} $ and $P = \begin{bmatrix}
                                                           0 & 0 \cr 0 & 1 
                                                          \end{bmatrix}.$ 
   Then  $${\rm Tr } \: [P|A- {\rm Tr} \: [PA]|^p] = 1 =  \min_{\lambda \in \mathbb{C}} \|A-\lambda\|.$$

\subsection{Remark} It would be interesting to know whether $\Phi \colon M_n(\mathbb{C}) \rightarrow M_k(\mathbb{C})$ is a positive unital linear map then  
the inequality 
 $$ \Phi(|A - \Phi(A)|^4)^{1/4} \leq {4 \over 3} \min_{\lambda \in \mathbb{C}} \|A - \lambda I\| $$
holds. The corresponding result for the noncommutative standard deviation was proved in \cite[Theorem 3.1]{BSh}.

\section{Moment inequalities for the tracial state}

We start this section with a moment estimate of matrices, determined by the tracial state. 

\subsection{Central moments for the tracial state}

Let $A$ denote an $n \times n$ matrix and let us define its Schatten $p$-norm
 $$ \|A\|_p= \left(\mbox{Tr } |A|^p \right)^{1/p},$$
where $1 \leq p <  \infty$ and $|A| = (A^*A)^{1/2}$ by definition. Then $\|\cdot\|_p$ is a norm on $M_n(\mathbb{C}).$ 
We recall that the duality formula 
 $$ \|A\|_p= \max \{ |\mbox{Tr } [B^*A]| \colon \|B\|_q \leq 1  \} $$
  holds with $1/p + 1/q = 1$ (\cite[Theorem 7.1]{C}). 
 
 Let $X$ be a (real) discrete random variable on a finite set $\{1, \hdots, n\}$. We recall that $\mathbb{E}(|X - \mathbb{E}(X)|^p)$ is the largest 
 if the underlying probability measure is concentrated on at most two atoms (for instance, see the proof \cite[p. 168]{L2}). Moreover,
 $\alpha\mu^{1/p}_p(X) = \alpha\mu^{1/p}_p(X-\beta),$ for any real $\alpha$ and $\beta,$ thus
  $$ \mathbb{E}(|X - \mathbb{E}(X)|^p)^{1/p} \leq b_p^{1/p} \max_{i,j} |X(i) - X(j)|, $$ where $b_p = \max_{0 \leq  t \leq 1} t^p(1-t) + t(1-t)^p.$
 
\begin{lemma} Let $A \in M_n(\mathbb{C})$ be a normal matrix and let $1 \leq p < \infty.$ Then
  $$ \left({\rm Tr} \: D\left|A - {\rm Tr } \: DA\right|^p\right)^{1/p} \leq 2b_p^{1/p} \min_{\lambda \in \mathbb{C}} \|A-\lambda I_n \|$$
  holds, where $b_p$ denotes the largest $p^{\rm th}$ central moment of the Bernoulli distribution. 
 \end{lemma} 
 
 \begin{proof}  By means of a diagonalization and the previous remarks, for any Hermitian $H$ and density $D$ one has that 
  \begin{equation*}
    ({\rm Tr}\left[D\left|H - {\rm Tr } \: DH\right|^p\right])^{1/p}  \leq  b_p^{1/p} \: {\rm diam } \: \sigma(H) = 2 b_p^{1/p}\min_{\lambda \in \mathbb{C}} \|H-\lambda\|.
  \end{equation*}
  For a normal $A$, $\min_{\lambda \in \mathbb{C}} \|A-\lambda I_n \|$ equals to the radius of the smallest enclosing circle of $\sigma(A).$ 
  Without loss of generality, we can assume that the center of this circle is at the origin. 
  Let us write $A$ as a diagonal matrix $A = \sum_{i=1}^n \lambda_i P_i,$ where $\lambda_i$-s are the eigenvalues of $A$ and $P_i$-s are orthogonal projections.
  Set the diagonal matrices
    $$ \tilde{A}= \begin{bmatrix}
                      \lambda_1 & & & & & \\ & \ddots & & & & \\& & \lambda_n & & & \\ &&& \hspace{-.3cm} - \overline{\lambda}_1 & & \\ &  & & &\ddots  & \\ &&& & & -\overline{\lambda}_n   
                    \end{bmatrix} \mbox{and } \tilde{H}= \begin{bmatrix}
                      |\lambda_1| & & & & & \\ & \ddots & & & & \\& & |\lambda_n| & & & \\ &&& - |\lambda_1| & & \\ &  & & &\ddots  & \\ &&& & &  \hspace{-.3cm}-|\lambda_n|   
                    \end{bmatrix}$$ in
   $M_{2n}(\mathbb{C}).$ Clearly, $\min_{\lambda \in \mathbb{C}} \|A-\lambda\|  =  \min_{\lambda \in \mathbb{C}} \|\tilde{H}-\lambda\|.$  
   Moreover, we obtain with $\tilde{D} = D \oplus  0 \in M_{2n}(\mathbb{C})$ that 
    \begin{eqnarray*}
     \begin{split}
         {\rm Tr}\left[D\left|A - {\rm Tr } \: DA\right|^p\right]^{1/p}  &=  {\rm Tr}\left[\tilde{D}|\tilde{A} - {\rm Tr } \: \tilde{D}\tilde{A}|^p\right]^{1/p} \\
        &= {\rm Tr}\left[\tilde{D}\left( \sum_{i=1}^n |\lambda_i|^p |(P_i \oplus -P_i) - \mbox{Tr}[\tilde{D}(P_i \oplus -P_i) ]|^{p}\right) \right]^{1/p} \\
        &=  {\rm Tr}\left[\tilde{D}|\tilde{H}- {\rm Tr } \: \tilde{D}\tilde{H}|^p\right]^{1/p} , \\
   &\hspace{-3.7cm} \mbox{ and since } \tilde{H} \mbox{ is Hermitian }     \\
        &\leq 2  b_p^{1/p} \min_{\lambda \in \mathbb{C}} \|\tilde{H}-\lambda \| \\
        &= 2  b_p^{1/p}  \min_{\lambda \in \mathbb{C}} \|A-\lambda \|,
     \end{split}
    \end{eqnarray*}
   which completes the proof.  
 \end{proof}

\begin{thm} Let $A \in M_n(\mathbb{C})$ and let $1 \leq p < \infty.$ Then
  $$ \left({1 \over n} {\rm Tr} \: \left|A -  {1 \over n} {\rm Tr } \: A\right|^p\right)^{1/p} \leq 2b_p^{1/p} \min_{\lambda \in \mathbb{C}} \|A-\lambda I_n \|$$
  holds, where $b_p$ denotes the largest $p^{\rm th}$ central moment of the Bernoulli distribution. 
 \end{thm}

 \begin{proof}
   Without loss of generality one can assume that $A$ is a contraction, i.e. $\|A\|=1.$
   From the singular value decomposition of $A$, one can find two unitaries $U_1$ and $U_2$ such that
   $$ A = {1\over 2} U_1 + {1\over 2} U_2 $$ (see \cite[p. 62-63]{Bh}, for instance).
   The convexity of the Schatten $p$-norms and the central moment estimates of normal matrices in Lemma 2 imply that  
    \begin{eqnarray*}
      \begin{split}
       \left({1 \over n} {\rm Tr} \: \left|A -  {1 \over n} {\rm Tr } \: A\right|^p\right)^{1/p} &\leq 
        {1\over 2}\left({1 \over n} {\rm Tr} \: \left|U_1 -  {1 \over n} {\rm Tr } \: U_1\right|^p\right)^{1/p} \\
        & \qquad +{1\over 2}\left({1 \over n} {\rm Tr} \: \left|U_2 -  {1 \over n} {\rm Tr } \: U_2\right|^p\right)^{1/p} \\
        &\leq b_p^{1/p}\|U_1\| + b_p^{1/p}\|U_2\|  \\
        &= 2b_p^{1/p}\|A\|. \\
       \end{split} 
    \end{eqnarray*}
  Changing $A$ to $A - \lambda I$ we get the proof of the statement. 
 \end{proof}
 
 An application of Hölder's inequality gives that the function $p \mapsto b_p^{1/p}$ is monotone increasing on $\mathbb{R}_+,$ hence
  $\lim_{p \rightarrow \infty} b_p^{1/p} = b_\infty = 1.$ Similarly, $\|\cdot\|_p \rightarrow \|\cdot\|$ follows for the Schatten $p$-norms, if $p \rightarrow \infty.$
  Therefore, we obtain with (2) at hand that 
 
\begin{cor}
  Let $A \in M_n(\mathbb{C})$ be a normal matrix. Then
  $$ {\sqrt{3} \over 2} \left\| A  - {1 \over n}{\rm  Tr } A \right\| \leq  \: {\rm spd }(A). $$
  Moreover, if $A$ is Hermitian then
   $$ \left\| A  - {1 \over n}{\rm  Tr } A \right\| \leq  \: {\rm spd }(A). $$
 \end{cor}
 
 \subsection{Central moments of matrix elements}
 
 In this section, we make some estimates of the moments of matrix elements.


 A conditional expectation operator $\mathbb{E}_{\mathfrak{B}}$ is an orthogonal projection from the matrix algebra $M_n(\mathbb{C}),$ endowed with the 
 Hilbert--Schmidt inner product $\langle A,B \rangle = {\rm Tr} \: B^*A$, onto the $*$-subalgebra $\mathfrak{B}$ (see \cite[Section 4.3]{C}). 
 Here we collect a few basic properties of the conditional expectation operator. First, we recall that for any 
 $A \in M_n(\mathbb{C}),$
  $$ \mbox{Tr } A =  \mbox{Tr } \mathbb{E}_{\mathfrak{B}}(A).$$
 Moreover, for each $B \in \mathfrak{B},$ it follows the module properties 
 $$ \mathbb{E}_{\mathfrak{B}}(BA) = B\mathbb{E}_{\mathfrak{B}}(A) \quad \mbox{ and } \quad \mathbb{E}_{\mathfrak{B}}(AB) = \mathbb{E}_{\mathfrak{B}}(A)B.  $$
 A useful property here is that the conditional expectation operators can be uniformly approximated by the convex sums
 of the unitary conjugates of $A$. That is, for all $\varepsilon > 0,$ 
 there  exist unitary operators $U_1, \hdots, U_m$ in the commutant algebra of $\mathfrak{B}$ 
 such that $$\| \mathbb{E}_\mathfrak{B}(A) - \sum_{j=1}^m \lambda_j U_j^* A U_j \| \leq \varepsilon, $$
  $\sum_{j=1}^m \lambda_j = 1$ and $0 < \lambda_1, \hdots, \lambda_m < 1.$
 For a proof of these statements, we refer the reader to \cite[Theorem 4.13]{C}.
  
 The following proposition might be known in the literature, however, we were unable to find any reference.
 
   \begin{prop}
   Let $\mathfrak{C}$ be a unital $*$-subalgebra of $M_n(\mathbb{C})$ 
   and let $\mathbb{E}_\mathfrak{C}$ be the conditional expectation operator onto $\mathfrak{C}.$
   Then 
          $${\rm Tr } \: |\mathbb{E}_\mathfrak{C}(A)|^p \leq {\rm Tr} \: |A|^p,$$
   for every $1 \leq p  < \infty.$         
 \end{prop}
    
  \begin{proof}
     The duality formula tells us  
     $$({\rm Tr } \: |\mathbb{E}_\mathfrak{C}(A)|^p)^{1/p} = \max_{B \in M_n(\mathbb{C})} \{ {|\rm Tr} \: [B\mathbb{E}_\mathfrak{C}(A)]| \colon \|B\|_q \leq 1 \}$$
     holds where $1/p + 1/q = 1.$
     Furthermore, for any $\varepsilon > 0,$ there exist unitary matrices $W_1, \hdots, W_m$ such that 
     $$\| \mathbb{E}_\mathfrak{C}(A) - \sum_{j=1}^m \lambda_j W_j^* A W_j \| \leq \varepsilon $$
     and $\sum_{j=1}^m \lambda_j = 1$ $(\lambda_j \geq 0).$
    Hence
   \begin{eqnarray*}
    \begin{split} 
      (\mbox{Tr} \: |\mathbb{E}_\mathfrak{C}(A)|^p)^{1/p}  &= \max_{B \in M_n(\mathbb{C})} \: \Bigl \{ \Bigl| \mbox{Tr} \: \Bigl[ \sum_{j=1}^m \lambda_j  BW_j^* A W_j \Bigr] \Bigr| \colon \|B\|_q \leq 1\Bigr\} + O(\varepsilon)\\
      &\leq    \max_{B \in M_n(\mathbb{C})}  \:\Bigl \{  \sum_{j=1}^m \lambda_j | \mbox{Tr} \:  [BW_j^* A W_j] |\colon \|B\|_q \leq 1\Bigr\} + O(\varepsilon), \\
      &\leq     \sum_{j=1}^m \lambda_j \max_{B \in M_n(\mathbb{C})}  \:\Bigl \{  | \mbox{Tr} \:  [BW_j^* A W_j] |\colon \|B\|_q \leq 1\Bigr\} + O(\varepsilon), \\
      &\hspace{-2.8cm} \mbox{and applying again the duality formula,} \\
      &= \sum_{j=1}^m \lambda_j   (\mbox{Tr} \:   |W_j^* A W_j|^p)^{1/p} + O(\varepsilon) \\
      &=   (\mbox{Tr} \:|A|^p)^{1/p} + O(\varepsilon),   
    \end{split}
   \end{eqnarray*} 
   which is what we intended to have.
  \end{proof}

 Now we can prove 
 
 \begin{thm}
   Let $e_1, \hdots, e_n$ be an orthonormal basis of $\mathbb{C}^n.$ For any $A \in M_n(\mathbb{C})$ and $1 \leq p < \infty,$ 
   \begin{eqnarray*}
    \begin{split}
    \left( {1 \over n}\sum_{i=1}^n |\langle Ae_i, e_i \rangle - {1 \over n}\sum_{j=1}^n \langle Ae_j, e_j \rangle|^p \right)^{1/p}&\leq 
       \left({1 \over n} {\rm Tr} \: \left|A -  {1 \over n} {\rm Tr } \: A\right|^p\right)^{1/p}  \\
       &\leq 2b_p^{1/p} \min_{c \in \mathbb{C}} \|A - \lambda I_n\|.  
    \end{split}   
   \end{eqnarray*} 
 \end{thm}
 
 \begin{proof}
   Let $\mathfrak{C}$ denote the commutative unital $*$-algebra 
   generated by the orthogonal projections $e_ie_i^*$ $(1 \leq i \leq n)$. 
   From the previous proposition one obtains that 
    \begin{eqnarray*}
        \sum_{i=1}^n |\langle Ae_i, e_i \rangle - {1 \over n}\sum_{j=1}^n \langle Ae_j, e_j \rangle|^p = {\rm Tr} \: \left|\mathbb{E}_{\mathfrak{C}}\left(A 
        -  {1 \over n} {\rm Tr } \: A\right)\right|^p  
        \leq  {\rm Tr} \: \left|A -  {1 \over n} {\rm Tr } \: A\right|^p,
    \end{eqnarray*}
     which is what we intended to have. 
 \end{proof}
 
 Lastly, the next corollary gives some information about the spread of Hermitians and normal matrices 
 in terms of the statistical dispersions of their diagonal elements.  
 
 \begin{cor}
  Let $1 \leq p < \infty.$  Let $A \in M_n(\mathbb{C})$ be a normal matrix. Then
  $$ \left({1 \over n}\sum_{i=1}^n |\langle Ae_i, e_i \rangle - {1 \over n}\sum_{j=1}^n \langle Ae_j, e_j \rangle|^p\right)^{1/p} \leq {2 \over \sqrt{3}}b_p^{1/p} \: {\rm spd }(A). $$
  Moreover, if $A$ is Hermitian then
   $$ \left({1 \over n}\sum_{i=1}^n |\langle Ae_i, e_i \rangle - {1 \over n}\sum_{j=1}^n \langle Ae_j, e_j \rangle|^p\right)^{1/p} \leq b_p^{1/p} \: {\rm spd }(A). $$
 \end{cor}

\end{document}